\documentclass[11pt]{amsart}

\usepackage{amsmath}
\usepackage{amssymb,mathtools,amsthm,latexsym}

\calclayout

\title{The extension property for domains with one singular point}
\author{Pekka Koskela and Zheng Zhu}

\address{Pekka Koskela\\
Department of Mathematics and Statistics\\
University of Jyv\"askyl\"a, P.O. Box 35 (MaD),
FI-40014, Jyv\"askyl\"a, Finland}
\email{\tt pekka.j.koskela@jyu.fi}

\address{Zheng Zhu\\
Department of Mathematics and Statistics\\
University of Jy\"askyl\"a, P.O. Box 35 (MaD),
FI-40014, Jyv\"askyl\"a, Finland}
\email{\tt zheng.z.zhu@jyu.fi}

\usepackage{stackrel}

\usepackage{color}

\numberwithin{equation}{section}
\long\def\colred#1\endred{{\color{red}#1}}
\long\def\colgreen#1\endgreen{{\color{green}#1}}
\long\def\colmagenta#1\endmagenta{{\color{magenta}#1}}
\long\def\colblue#1\endblue{{\color{blue}#1}}
\long\def\colyellow#1\endyellow{{\color{yellow}#1}}

\usepackage{tikz}
\usetikzlibrary{matrix,arrows}

\theoremstyle{plain}
\newtheorem{theorem}[equation]{Theorem}
\newtheorem{proposition}[equation]{Proposition}
\newtheorem{lemma}[equation]{Lemma}

\newtheorem{conjecture}[equation]{Conjecture}

\theoremstyle{remark}

\theoremstyle{definition}

\newtheorem{defn}[equation]{Definition}

\newtheorem*{question*}{Question}

\usepackage{graphicx}


\thanks{The authors have been supported  by the Academy of Finland via Centre of Excellence in Analysis and Dynamics Research (Project \#323960).}

\newcounter{prob}
\def\rr{{\mathbb R}}
\def\rn{{{\rr}^n}}

\def\fz{\infty}

\def\boz{{\Omega}}

\def\bint{{\ifinner\rlap{\bf\kern.25em--}
\int\else\rlap{\bf\kern.45em--}\int\fi}\ignorespaces}

\def\bbint{{\ifinner\rlap{\bf\kern.25em--}
\hspace{0.078cm}\int\else\rlap{\bf\kern.45em--}\int\fi}\ignorespaces}

\def\wp{{W^{1,p}(\boz)}}

\def\lp{{L^{p}(\boz)}}

\def\r{\right}
\def\lf{\left}

\def\setcolon{;}

\newcommand{\R}{\ensuremath{\mathbb{R}}}

\def\XXint#1#2#3{{\setbox0=\hbox{$#1{#2#3}{\int}$ }
\vcenter{\hbox{$#2#3$ }}\kern-.58\wd0}}

\def\vint_#1{\mathchoice%
        {\mathop{\kern 0.2em\vrule width 0.6em height 0.69678ex depth -0.58065ex
                \kern -0.8em \intop}\nolimits_{\kern -0.4em#1}}%
        {\mathop{\kern 0.1em\vrule width 0.5em height 0.69678ex depth -0.60387ex
                \kern -0.6em \intop}\nolimits_{#1}}%
        {\mathop{\kern 0.1em\vrule width 0.5em height 0.69678ex
            depth -0.60387ex
                \kern -0.6em \intop}\nolimits_{#1}}%
        {\mathop{\kern 0.1em\vrule width 0.5em height 0.69678ex depth -0.60387ex
                \kern -0.6em \intop}\nolimits_{#1}}}
\def\vintslides_#1{\mathchoice%
        {\mathop{\kern 0.1em\vrule width 0.5em height 0.697ex depth -0.581ex
                \kern -0.6em \intop}\nolimits_{\kern -0.4em#1}}%
        {\mathop{\kern 0.1em\vrule width 0.3em height 0.697ex depth -0.604ex
                \kern -0.4em \intop}\nolimits_{#1}}%
        {\mathop{\kern 0.1em\vrule width 0.3em height 0.697ex depth -0.604ex
                \kern -0.4em \intop}\nolimits_{#1}}%
        {\mathop{\kern 0.1em\vrule width 0.3em height 0.697ex depth -0.604ex
                \kern -0.4em \intop}\nolimits_{#1}}}

\begin{document}

\maketitle

\begin{abstract}
An arbitrary outward cuspidal domain is shown to be bi-Lipschitz equivalent to a Lipschitz outward cuspidal domain via a global transformation. This allows us to extend earlier Sobolev extension results on Lipschitz outward cuspidal domains from the work of Maz'ya and Poborchi to arbitrary outward cuspidal domains. We also establish a limit case of extension results on outward cuspidal domains.
\end{abstract}

\section{Introduction}
A domain $\boz\subset\rn$ is said to be a Sobolev $(p, q)$-extension domain for $1\leq q\leq p\leq\fz$, if there exists a bounded extension operator
$$ E:W^{1,p}(\boz)\to W^{1,q}(\rn),$$
such that for every $u\in W^{1, p}(\boz)$, we have $E(u)\in W^{1, q}(\rn)$ with 
\[\|E(u)\|_{W^{1, q}(\rn)}\leq C\|u\|_{W^{1,p}(\boz)}\]
for a constant $C$ independent of $u$. The smallest constant in the inequality above is denoted by $\|E\|$. In \cite{C61, Stein}, Calder\'on and Stein proved that if $\Omega\subset\mathbb R^n$ is a Lipschitz domain, then there exists a bounded linear extension operator $E: W^{k, p}(\Omega)\to W^{k, p}(\mathbb R^n)$, for all $k\geq 1$ and $1\leq p\leq \infty$. Here $W^{k, p}(\boz)$ is the class of those $L^p$-integrable functions whose weak derivatives up to order $k$ belong to $L^p(\Omega)$. In \cite{J81}, Jones introduced the notion of $(\varepsilon,\delta)$-domains which are generalizations of Lipschitz domains. He proved that, for every $(\varepsilon,\delta)$-domain, there exists a bounded linear extension operator $E: W^{k, p}(\Omega)\to W^{k, p}(\mathbb R^n)$, for all $k\geq 1$  and $1\leq p\leq\fz$. This has motivated the search for geometric characterizations for Sobolev extension domains. A geometric characterization of simply connected planar Sobolev $(2,2)$-extension domains was obtained in \cite{VGL79}. By a more recent sequence of results in \cite{pekkaJFA, KRZ1, KRZ2, ShJFA}, we understand the geometry of simply connected planar Sobolev $(p, p)$-extension domains, for all $1\leq p\leq\fz$. Furthermore, by \cite{ShZ16}, geometric characterizations of planar simply connected extension domains are also known in the case of homogeneous Sobolev spaces $L^{k, p}(\Omega)$ for $2<p<\infty$. Here $L^{k, p}(\boz)$ is the seminormed space of those locally integrable functions whose $k$-th order distributional partial derivatives belong to $L^p(\boz)$. However, no characterizations are available in the general setting. 

In this paper, we consider Sobolev extension properties for a class of quasiconvex Euclidean domains with only a single singular boundary point.  A domain $\boz\subset\rn$ is called quasiconvex, if for every $x, y\in\boz$ there exists a curve $\gamma_{x,y}\subset\boz$ connecting $x, y$ with $${\rm length}(\gamma_{x, y})\leq C|x-y|$$ for a constant $C$ independent of $x$ and $y$. We study the outward cuspidal domains defined by setting
\begin{equation}\label{cusp}
\boz^n_\psi:=\lf\{z=(t, x)\in (0, 1]\times\rr^{n-1}\setcolon |x|<\psi(t) \r\}\cup\{z=(t, x)\in[1, 2)\times\rr^{n-1}\setcolon |x|<\psi(1)\},
\end{equation}
where $\psi\colon (0,1]\to (0,\fz)$ is a left-continuous and increasing function. (Left-continuity is required just to ensure $\boz^n_\psi$ to be open. The term ``increasing'' is used in the non-strict sense.) If the left-continuous and increasing function $\psi$ is Lipschitz, $\boz_\psi^n$ is called a Lipschitz outward cuspidal domain. This kind of model domains have been widely studied, see Maz'ya and Poborchi's monograph \cite{Mazya} and references therein. From now on, every left-continuous and increasing function $\psi:(0, 1]\to(0, \fz)$ will be called a cuspidal function. Our class of domains $\boz_\psi^n$ was introduced in \cite{SPJZ}. It was shown in \cite{SPJZ} that for an arbitrary cuspidal function $\psi$, the Sobolev space $W^{1,p}(\boz_\psi^n)$ coincides with the Haj\l{}asz-Sobolev space $M^{1, p}(\boz_\psi^n)$ for all $1<p\leq\fz$. See \cite{Pitor} for the definition of the Haj\l{}asz-Sobolev space $M^{1, p}(\boz)$. 
\begin{figure}[h!tbp]
\centering
\includegraphics[width=0.5\textwidth]
{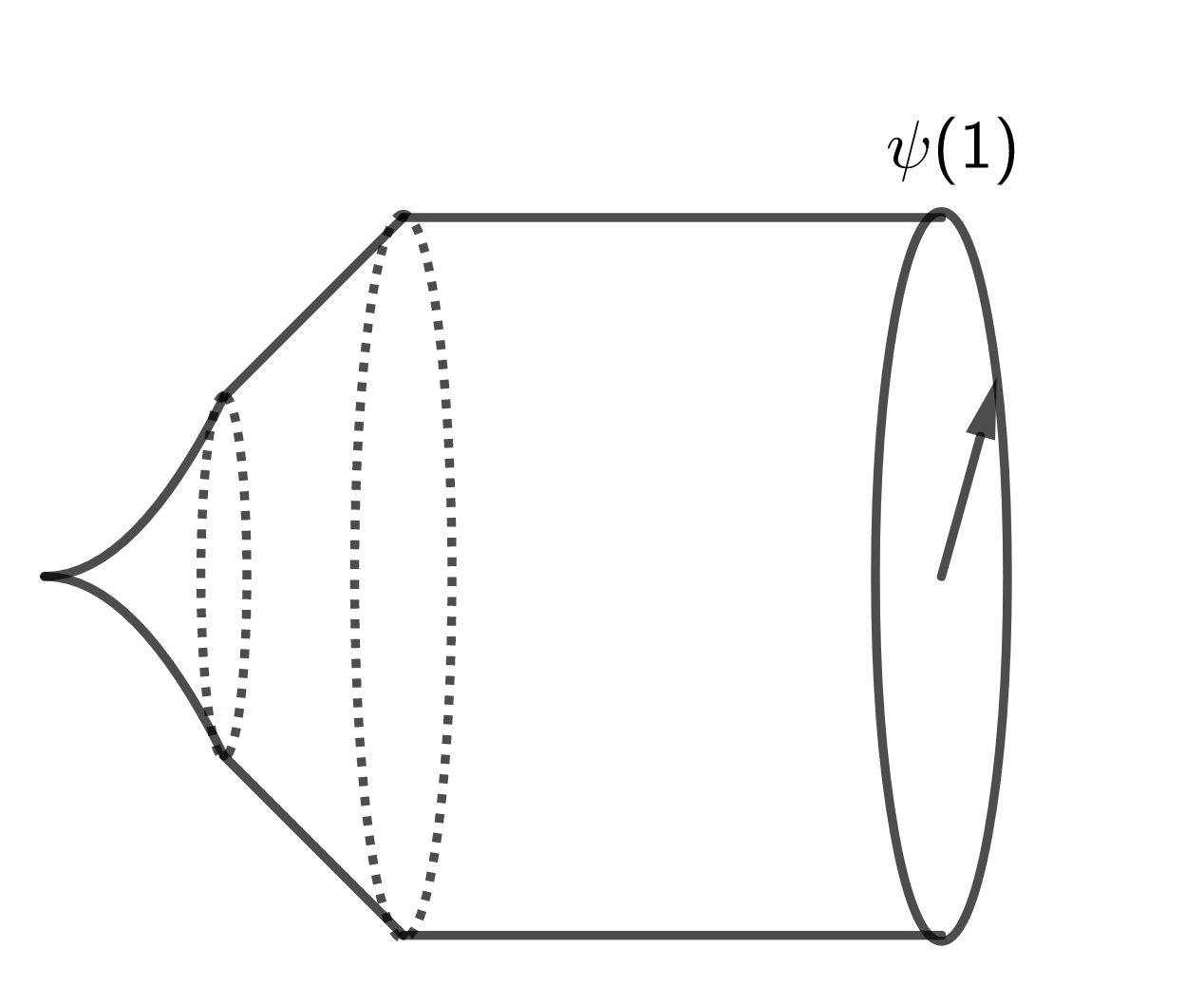}\label{fig:cusps}
\caption{An outward cuspidal domain $\boz_\psi^n$}
\end{figure}
Returning to the problem of geometric characterization for Sobolev extension domains, one may ask a natural question:
\begin{center}
\textbf{For which cuspidal functions $\psi$, the outward cuspidal domain $\boz_\psi^n$ is a Sobolev $(p, q)$-extension domain for given $1\leq q\leq p\leq\fz$?}
\end{center}
The first observation is that $\boz_\psi^n$ is a Sobolev $(\fz, \fz)$-extension domain, for every cuspidal function $\psi$. This follows from a result in \cite{hkt} since $\boz_\psi^n$ is always quasiconvex for an arbitrary cuspidal function $\psi$.  

In \cite{Mazya1, Mazya2, Mazya}, for certain Lipschitz cuspidal functions $\psi$, Maz'ya and Poborchi used integrability conditions on $\psi$ to characterize the Sobolev extension property for $\boz_\psi^n$. By transferring an outward cuspidal domain onto a Lipschitz outward cuspidal domain via a global bi-Lipschitz transformation, we obtain a more general version of their result. 
\begin{theorem}\label{prop:mazya}
Let $\psi:(0, 1]\to(0, \fz)$ be a cuspidal function such that $\psi\big|_{(0, 1]}>0$ and the function $\psi(t)/t$ is nondecreasing on $(0, 1]$ with $\lim_{t\to0}\psi(t)/t=0$. 
We have following statements.\\
{\textbf{$(1):$}} If 
\begin{equation}\label{eq:inc1}
\int_0^1\lf(\frac{t^s}{\psi(t)}\r)^{\frac{n}{s-1}}\frac{dt}{t}<\fz, 
\end{equation}
then there exists a bounded linear extension operator $E_1$ from $W^{1,p}(\boz^n_\psi)$ to $W^{1, q}(\rn)$ whenever $\frac{1+(n-1)s}{n}\leq p<\fz$ and $1\leq q\leq\frac{np}{1+(n-1)s}$, and a bounded linear extension operator $E_2$ from $W^{1, p}(\boz^n_\psi)$ to $W^{1, q}(\rn)$ whenever $\frac{1+(n-1)s}{2+(n-2)s}\leq p<\fz$ and $1\leq q\leq\frac{(1+(n-1)s)p}{1+(n-1)s+(s-1)p}$.\\
{\textbf{$(2):$}} If  
\begin{equation}\label{eq:inc2}
\int_0^1\lf(\frac{t^s}{\psi(t)}\r)^{\frac{n}{s-1}}\lf|\log\lf(\frac{\psi(t)}{t}\r)\r|^{-\alpha}\frac{dt}{t}<\fz
\end{equation}
with $\alpha=\frac{(n-2)p}{p+1-n}$, then there exists a bounded linear extension operator $E_3$ from $W^{1, p}(\boz^n_\psi)$ to $W^{1, q}(\rn)$ whenever $\frac{(n-1)^2s+(n-1)}{n}\leq p<\fz$ and $1\leq q\leq n-1$.

Furthermore, under the doubling condition
\begin{equation}\label{eq:doubling}
 \psi(2t)\leq C\psi(t),\ {\rm for}\ t\in\lf(0, \frac{1}{2}\r),
 \end{equation}
on $\psi$, the statements converse to $(1)$ and $(2)$ also hold.
\end{theorem}

Given $1<s<\fz$, we refer to the domain $\boz_\psi^n$ with $\psi(t)=t^s$ by $\boz_{t^s}^n$. Let $n\geq 3$. Theorem \ref{prop:mazya} yields that $\boz^n_{t^s}$ is a Sobolev $(p, q)$-extension domain, whenever $1\leq q<n-1$ and $(n-1)q/(n-1-q)\leq p\leq\fz$. Our bi-Lipschitz transformation method allows us to extend this result from the case of $t^s$ to arbitrary cuspidal functions. This result can be regarded as a limit case of Theorem \ref{prop:mazya} and also of theorems by Maz'ya and Poborchi in \cite{Mazya}.
\begin{theorem}\label{thm:main}
Let $3\leq n<\fz$ and $\psi:(0, 1]\to(0, \fz)$ be a cuspidal function. Then the corresponding outward cuspidal domain $\boz^n_\psi$ is a Sobolev $(p, q)$-extension domain, whenever $1\leq q<n-1$ and $(n-1)q/(n-1-q)\leq p\leq\fz$. 
\end{theorem}

The sharpness part of Theorem \ref{prop:mazya} also yields the sharpness of Theorem \ref{thm:main}.
\begin{proposition}\label{proposition}
For arbitrary $n-1\leq p<\fz$, there exists $1<s_1<\fz$ such that $\boz_{t^s}^n$ is not a Sobolev $(p, n-1)$-extension domain when $s>s_1$. For arbitrary $1\leq q<n-1$ and $q\leq p<(n-1)q/(n-1-q)$, there exists $1<s_2<\fz$ such that $\boz_{t^s}^n$ is not a Sobolev $(p, q)$-extension domain when $s>s_2$.
\end{proposition}

 
The paper is organized as follows. Section $2$ contains definitions and preliminary results. Section $3$ contains proofs of all results presented above. Section $4$ contains some further discussion and a conjecture.
\section{Definitions and Preliminaries}
In this note, $\boz\subset\R^n$ is always a bounded domain. $C$ will refer to constants that depend on various parameters and may differ even in a chian of inequalities. 
The Euclidean distance between points $x,y\in\rn$ is denoted by $|x-y|$. The open $n$-dimensional ball of radius $r$ centered at the point $x$ is denoted by $B^{n}(x,r)$. 

 Let us give the definition of the Sobolev space $W^{1,p}(\boz)$.
\begin{defn}\label{Sobolev}
We define the first order Sobolev space $W^{1,p}(\boz)$, $1\leq p\leq \fz$, as the set 
\begin{equation}
\left\{\, u\in\lp \setcolon \nabla u\in L^p(\boz\setcolon\rn)\, \right\}\, .\nonumber
\end{equation}
Here $\nabla u=\left(\frac{\partial u}{\partial x_1}\, ,\, \dots \, ,\,\frac{\partial u}{\partial x_n}\right)$ is the weak (or distributional) gradient of the integrable function $u$. 
\end{defn}
The Sobolev space $W^{1,p}(\boz)$ is equipped with the norm:
\begin{equation}
\|u\|_{\wp}=\|u\|_{L^p(\boz)}+\||\nabla u|\|_{L^p(\boz)}\nonumber
\end{equation}
for $1\leq p\leq\fz$, where $\|f\|_{L^p(\boz)}$ denotes the usual $L^p$-norm for $p \in [1,\infty]$. Let us give the definition of Sobolev extension domains.
\begin{defn}\label{de:extdo}
Let $1\leq q\leq p\leq\fz$. A bounded domain $\boz\subset\rn$ is said to be a Sobolev $(p, q)$-extension domain, if there is a bounded extension operator $E$ from $W^{1,p}(\boz)$ to $W^{1,q}(\rn)$ such that, for every $u\in W^{1,p}(\boz)$, there exists a function $E(u)\in W^{1,q}(\rn)$ with $E(u)\big|_\boz\equiv u$ and 
\[\|E(u)\|_{W^{1,q}(\rn)}\leq C\|u\|_{W^{1,p}(\boz)}\]
for a positive constant $C$ independent of $u$.
\end{defn}

An outward cuspidal domain $\boz_\psi^n$ has a singular point on the boundary. However, it still has some nice geometric properties. For example, it satisfies the following $segment\ condition$.
\begin{defn}\label{defn:segment}
We say that a domain $\boz\subset\rn$ satisfies the segment condition if every $x\in\partial\boz$ has a neighborhood $U_x$ and a nonzero vector $y_x$ such that if $z\in\overline\boz\cap U_x$, then $z+ty_x\in\boz$ for $0<t<1$.
\end{defn}
The following lemma tells us that Sobolev functions on a domain with the segment condition can be approximated by globally smooth functions. See \cite[Theorem 3.22]{adams}.
\begin{lemma}\label{lem:density}
If a domain $\boz\subset\rn$ satisfies the segment condition, then the set of restrictions to $\boz$ of functions in $C_o^\fz(\rn)$ is dense in $W^{1,p}(\boz)$ for $1\leq p<\fz$. In short, $C_o^\fz(\rn)\cap W^{1,p}(\boz)$ is dense in $W^{1,p}(\boz)$ for $1\leq p<\fz$.
\end{lemma}

\section{Proofs of Theorem \ref{prop:mazya}, Theorem \ref{thm:main} and Proposition \ref{proposition}}
As mentioned in the introduction, every outward cuspidal domain $\boz_\psi^n$ is a Sobolev $(\fz, \fz)$-extension domain. We show that every outward cuspidal domain is globally bi-Lipschitz equivalent to a Lipschitz outward cuspidal domain. For the proof of this result, we first introduce a Lipschitz cuspidal function $\hat\psi$ generated by a given cuspidal function $\psi$.
\begin{lemma}\label{le:LipcF}
Let $\psi:(0, 1]\to(0, \fz)$ be an arbitrary cuspidal function. Then the following claims hold.\\
$(1):$ For every $0<\hat t<1$, there exists a unique pair $(t_{\hat t}, r_{\hat t})$ with $0<t_{\hat t}<1$, $\psi(t_{\hat t})\leq r_{\hat t}\leq\lim_{s\to t_{\hat t}^+}\psi(s)$ and $t_{\hat t}+r_{\hat t}=(1+\psi(1))\hat t$.\\
$(2):$ The function $\hat\psi:(0, 1]\to(0, \fz)$ defined by setting 
\[\hat\psi(\hat t)=r_{\hat t}\ {\rm for\ every}\ \hat t\in(0, 1]\]
is a Lipschitz cuspidal function.
\end{lemma}
\begin{figure}[h!tbp]
\centering
\includegraphics[width=0.4\textwidth]
{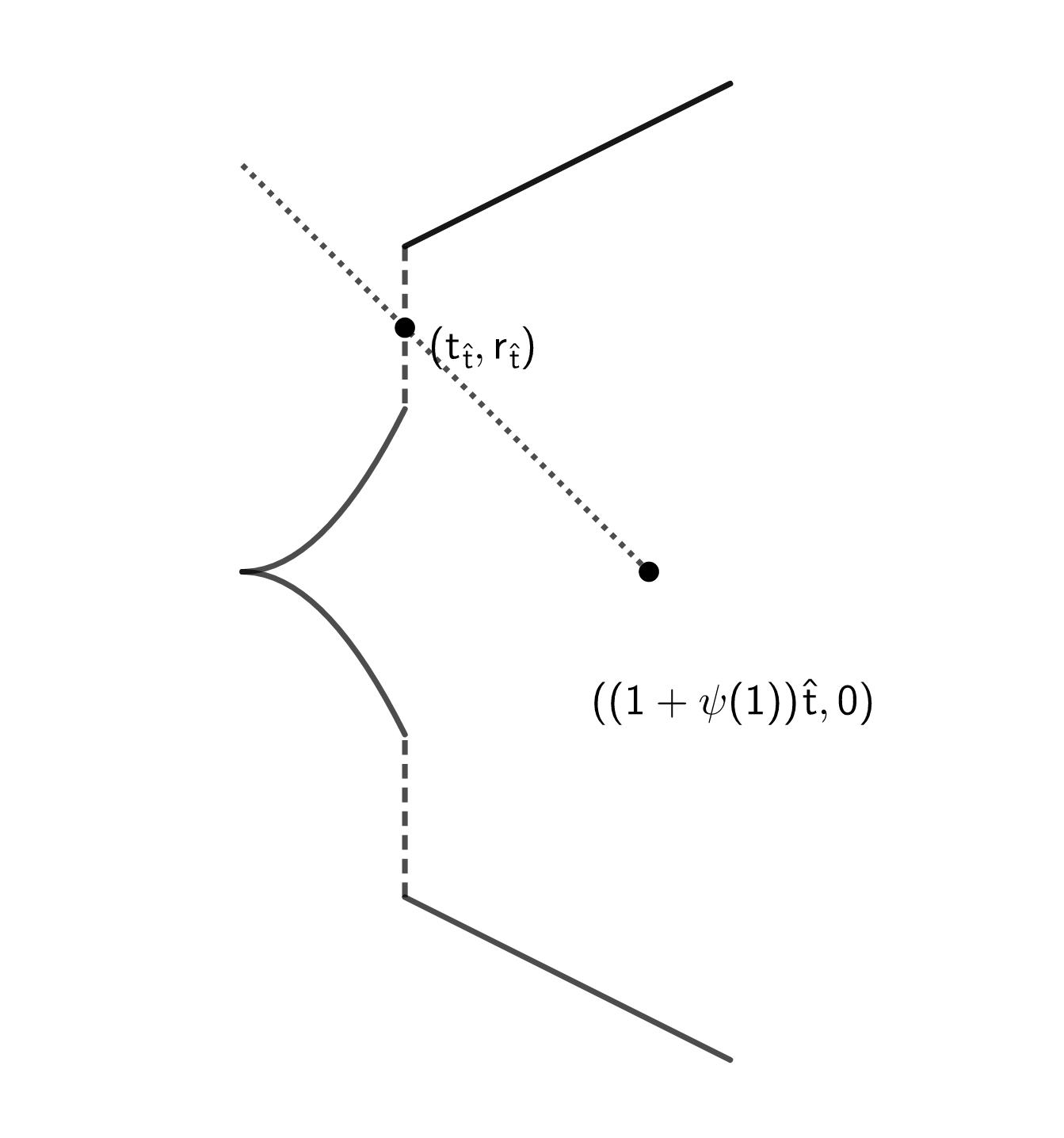}\label{fig:Lip}
\caption{Bi-Lipschitz transformations}
\end{figure}
\begin{proof}
Let $(t, r)\in(0, 1]\times(0, \psi(1))$ be a pair of positive numbers such that for every $x\in\rr^{n-1}$ with $|x|=r$, we have $(t, x)\in\partial\boz^n_\psi$. Define a function $T$ on $\partial\boz^n_\psi$ by setting
\[T(t, x):=t+|x|\]
for every $(t, x)\in\partial\boz^n_\psi$. Since $\psi$ is left-continuous and increasing on $(0, 1]$, we always have $T(t_1,x_1)<T(t_2,x_2)$ for every $(t_1, x_1), (t_2, x_2)\in\partial\boz^n_\psi$ with $0<t_1<t_2<1$. By the same reason, we have  $T(t, x_1)<T(t, x_2)$ for every $(t, x_1), (t, x_2)\in\partial\boz^n_\psi$ with $0<t<1$ and $$\psi(t)\leq|x_1|<|x_2|\leq\lim_{s\to t^+}\psi(s).$$ Hence, for every $0<\hat t<1$, there exists a unique pair $(t_{\hat t}, r_{\hat t})$ with $0<t_{\hat t}<1$, $\psi(t_{\hat t})\leq r_{\hat t}\leq\lim_{s\to t_{\hat t}^+}\psi(s)$ and $t_{\hat t}+r_{\hat t}=(1+\psi(1))\hat t$. Now, let us show $\hat\psi$ is a Lipschitz cuspidal function. By the definition, it is easy to see $\hat\psi$ is increasing. Hence, it suffices to show that it is Lipschitz. Define a measurable subset $A\subset(0, 1]$ by setting
\[A:=\{\hat t\in(0, 1]: r_{\hat t}\geq t_{\hat t}\}.\] 
Let $\hat t_1, \hat t_2\in(0, 1]$ be arbitrary. According to their locations, we divide the following argument into three cases. First, let us assume $\hat t_1, \hat t_2\in A$. Since $\hat t\sim r_{\hat t}$ for every $\hat t\in A$, by the definition, we have 
\begin{equation}\label{eq:In1}
|\hat\psi(\hat t_1)-\hat\psi(\hat t_2)|=|r_{\hat t_1}-r_{\hat t_2}|\leq C|\hat t_1-\hat t_2|.
\end{equation}
Next, we assume $\hat t_1, \hat t_2\in (0, 1]\setminus A$. For every $\hat t\in(0, 1]\setminus A$, we have $\hat t\sim t_{\hat t}$. Hence, the triangle inequality implies
\begin{equation}\label{eq:In2}
|\hat\psi(\hat t_1)-\hat\psi(\hat t_2)|=|r_{\hat t_1}-r_{\hat t_2}|\leq C|\hat t_1-\hat t_2|+C|t_{\hat t_1}-t_{\hat t_2}|\leq C|\hat t_1-\hat t_2|.
\end{equation}
Finally, we assume $\hat t_1\in A$ and $\hat t_2\in (0, 1]\setminus A$. Since both $t_{\hat t}$ and $r_{\hat t}$ are continuous with respect to $\hat t$, there exists $\hat s\in(0, 1]$ between $\hat t_1$ and $\hat t_2$ with $t_{\hat s}=r_{\hat s}$. By (\ref{eq:In1}) and (\ref{eq:In2}), the triangle inequality implies
\begin{equation}\label{eq:In3}
|\hat\psi(\hat t_1)-\hat\psi(\hat t_2)|\leq|\hat\psi(\hat t_1)-\hat\psi(\hat s)|+|\hat\psi(\hat s)-\hat\psi(\hat t_2)|\leq C|\hat t_1-\hat s|+C|\hat s-\hat t_2|\leq C|\hat t_1-\hat t_2|.
\end{equation}
By combining (\ref{eq:In1}), (\ref{eq:In2}) and (\ref{eq:In3}), we conclude that $\hat\psi$ is Lipschitz.
\end{proof}
We are ready to prove the following proposition which claims that an arbitrary outward cuspidal domain is globally bi-Lipschitz equivalent to a Lipschitz outward cuspidal domain.
\begin{proposition}\label{prop:bilip}
For a given cuspidal function $\psi$, let $\hat\psi$ be the Lipschitz cuspidal function defined in Lemma \ref{le:LipcF}. Then there exists a global bi-Lipschitz transformation $\mathcal O:\rn\to\rn$ with $\mathcal O(\boz^n_\psi)=\boz^n_{\hat\psi}$.
\end{proposition}
\begin{proof}
Let $\psi:(0, 1]\to(0, \fz)$ be a cuspidal function. Let $\hat\psi:(0, 1]\to(0, \fz)$ be the corresponding Lipschitz cuspidal function defined in Lemma \ref{le:LipcF}. We extend the definition of $\psi$ to the entire real line $\rr$ by setting $\psi(t)=0$ for every $t\leq 0$. Without loss of generality, we may assume $2\psi(1)<1$. Otherwise, we consider the cuspidal function $\tilde\psi$ defined by setting
\[\tilde\psi(t):=\frac{\psi(t)}{2\psi(1)}\ {\rm for}\ t\in(-\fz, 1).\]
Obviously, $\boz^n_\psi$ and $\boz^n_{\tilde\psi}$ are globally bi-Lipschitz equivalent.

We define
\[U_1:=\lf\{(t, x)\in (-\fz, 1+\psi(1)]\times\rr^{n-1}: |x|\leq 1+\psi(1)-t\r\}, U_2:=\boz_\psi^n\setminus U_1,\]
\[U_3:=\lf\{(t, x)\in(-\fz, \fz)\times\rr^{n-1}:|x|\geq\max\{\psi(1), 1+\psi(1)-t\}\r\},\]
\[U_4:=\lf\{(t, x)\in[2, \fz)\times\rr^{n-1}: |x|\leq\psi(1)\r\}.\]
Obviously, we have $\rn=\bigcup_{i=1}^4U_i$.  
Then we define a global transformation $\mathcal O:\rn\to\rn$ by setting 
\begin{equation}\label{eq:bilip}
\mathcal O(t, x):=\begin{cases}
\lf(\lf(\frac{1}{1+\psi(1)}\r)(t+|x|), x\r), & (t, x)\in U_1,\\
\lf(\lf(\frac{t}{1+|x|-\psi(1)}+\frac{2(|x|-\psi(1))}{1+|x|-\psi(1)}\r), x\r), & (t, x)\in U_2,\\
\lf(t+|x|-\psi(1), x\r), & (t, x)\in U_3,\\
\lf(t, x\r), & (t, x)\in U_4
\,. \end{cases}
\end{equation}
The global homeomorphism $\mathcal O$ is differentiable almost everywhere and there exists a positive constant $C>1$ such that for almost every $z\in\rn$, we have 
\[\frac{1}{C}\leq|D\mathcal O(z)|\leq C.\]
Hence, $\mathcal O:\rn\to\rn$ is a global bi-Lipschitz transformation. By a simple computation, we obtain 
\[\mathcal O(U_1\cap\boz^n_{\psi})=\{(s, y)\in(0, 1]\times\rr^{n-1}: |y|<\hat\psi(s)\}\]
and 
\[\mathcal O(U_2)=\{(s, y)\in(1, 2)\times\rr^{n-1}:|y|<\hat\psi(1)\}.\]
Hence $\mathcal O(\boz^n_\psi)=\boz^n_{\hat\psi}$.
\end{proof}
Let us prove Theorem \ref{prop:mazya}.
{\begin{proof}[Proof of Theorem \ref{prop:mazya}]
Let $\psi:(0, 1]\to(0, \fz)$ be a cuspidal function such that $\psi(t)/t$ is nondecreasing in $(0, 1]$ with $\lim_{t\to 0}\psi(t)/t=0$. Let $\hat\psi$ be the corresponding Lipschitz cuspidal function, defined as in Lemma \ref{le:LipcF}. Following the notation from Lemma \ref{le:LipcF}, for every $\hat t\in(0, 1]$, we have \[t_{\hat t}+r_{\hat t}=(1+\psi(1))\hat t\ {\rm and}\ \hat\psi(\hat t)=r_{\hat t}.\] 
Since $\psi$ is left-continuous and increasing, for every $\hat t\in\lf(0, \frac{1}{1+\psi(1)}\r)$, we have 
\[\psi(t_{\hat t})\leq r_{\hat t}<\psi\lf((1+\psi(1))\hat t\r).\]
 Hence, we have 
$$\lim_{\hat t\to0}\frac{\hat\psi(\hat t)}{\hat t}\leq C\lim_{\hat t\to 0}\frac{\psi(\hat t)}{\hat t}=0$$
 and
 \begin{equation}\label{eq:100}
\frac{\hat\psi(\hat t)}{\hat t}=\frac{r_{\hat t}}{\hat t}=\frac{(1+\psi(1))\hat t-t_{\hat t}}{\hat t}=(1+\psi(1))\lf(1-\frac{t_{\hat t}}{t_{\hat t}+r_{\hat t}}\r).
\end{equation}
By (\ref{eq:100}), to show that $\hat\psi(\hat t)/\hat t$ is increasing, it suffices to show that 
\begin{equation}\label{eq:increasing}
\frac{r_{\hat t_1}}{t_{\hat t_1}}\leq\frac{ r_{\hat t_2}}{t_{\hat t_2}}\ \ {\rm for}\ \  0<\hat t_1\leq\hat t_2<1.
\end{equation}
Since $\psi$ is non-decreasing and $\hat t_1\leq\hat t_2$, we always have $r_{\hat t_1}\leq r_{\hat t_2}$. Hence, if $t_{\hat t_1}=t_{\hat t_2}$, we immediately obtain inequality (\ref{eq:increasing}). If $t_{\hat t_1}<t_{\hat t_2}$, the fact that $\psi$ is non-decreasing implies $r_{\hat t_1}\leq\psi(t_{\hat t_2})\leq r_{\hat t_2}$. Since $\psi$ has at most countably many points of discontinuity, for every $\epsilon>0$, we can find a point $t_\epsilon\in(t_{\hat t_1}, t_{\hat t_1}+\epsilon)$ of continuity of $\psi$  with $r_{\hat t_1}\leq \psi(t_\epsilon)$. The fact that $\psi(t)/t$ is non-decreasing implies 
\[\frac{r_{\hat t_1}}{t_{\hat t_1}}\leq\lim_{\epsilon\to0}\frac{\psi(t_\epsilon)}{t_{\hat t_1}+\epsilon}\leq\lim_{\epsilon\to0}\frac{\psi(t_\epsilon)}{t_\epsilon}\leq\frac{\psi(t_{\hat t_2})}{t_{\hat t_2}}\leq\frac{r_{\hat t_2}}{t_{\hat t_2}}.\]

\textbf{$(1)$:} We show that the assumption that original cuspidal function $\psi$ satisfies the integrability condition (\ref{eq:inc1}) implies that the corresponding Lipschitz cuspidal function $\hat\psi$ also satisfies the integrability condition (\ref{eq:inc1}). Define a measurable subset $B\subset (0, 1]$ by setting 
\[B:=\{\hat t\in(0, 1]: t_{\hat t}\geq r_{\hat t}\}.\]
Then, for every $\hat t\in B$, we have $\hat t\sim t_{\hat t}$ and $\hat\psi(\hat t)=r_{\hat t}\geq\psi(t_{\hat t})$. Hence, we have 
\begin{equation}\label{eq:inc3}
\int_B\lf(\frac{\hat t^s}{\hat\psi(\hat t)}\r)^{\frac{n}{s-1}}\frac{d\hat t}{\hat t}\leq C\int_0^1\lf(\frac{t^s_{\hat t}}{\psi(t_{\hat t})}\r)^{\frac{n}{s-1}}\frac{dt_{\hat t}}{t_{\hat t}}<\fz.
\end{equation}
If $\hat t\in(0, 1]\setminus B$, we have $\hat\psi(\hat t)=r_{\hat t}\sim\hat t$. Then we have 
\begin{equation}\label{eq:inc4}
\int_{(0, 1]\setminus B}\lf(\frac{\hat t^s}{\hat\psi(\hat t)}\r)^{\frac{n}{s-1}}\frac{d\hat t}{\hat t}\leq C\int_0^1\hat t^{n-1}d\hat t<\fz.
\end{equation}
Hence, by combining the last two inequalities, we obtain that $\hat\psi$ satisfies the integrability condition that
\[\int_0^1\lf(\frac{\hat t^s}{\hat\psi(\hat t)}\r)^{\frac{n}{s-1}}\frac{d\hat t}{\hat t}<\fz.\]
In conclusion, we have shown that $\hat\psi$ satisfies all assumptions of the theorems due to Maz'ya and Poborchi in \cite[page 304 and 312]{Mazya}. Hence there exists a bounded linear extension operator $\widetilde E_1: W^{1, p}(\boz^n_{\hat\psi})\to W^{1, q}(\rn)$ whenever $\frac{1+(n-1)s}{n}\leq p<\fz$ and $1\leq q\leq\frac{np}{1+(n-1)s}$. Then, for $\frac{1+(n-1)s}{n}\leq p<\fz$, we define an extension operator $E_1$ on $W^{1,p}(\boz^n_\psi)$ by setting 
\[E_1(u)(x):=\widetilde E_1\lf(u\circ\mathcal O^{-1}\r)(\mathcal O(x))\]
for every function $u\in W^{1,p}(\boz^n_\psi)$ and every $x\in\rn$. By the facts that a bi-Lipschitz transformation preserves first order Sobolev spaces and that $\widetilde E_1: W^{1, p}(\boz^n_{\hat\psi})\to W^{1, q}(\rn)$ is a bounded linear extension operator whenever $\frac{1+(n-1)s}{n}\leq p<\fz$ and $1\leq q\leq\frac{np}{1+(n-1)s}$, we obtain that $E_1:W^{1, p}(\boz^n_\psi)\to W^{1, q}(\rn)$ is also a bounded linear extension operator whenever $\frac{1+(n-1)s}{n}\leq p<\fz$ and $1\leq q\leq\frac{np}{1+(n-1)s}$. The theorem from \cite[page 312]{Mazya} tells us that there exists a bounded linear extension operator $\widetilde E_2:W^{1, p}(\boz^n_{\hat\psi})\to W^{1, q}(\rn)$ whenever $\frac{1+(n-1)s}{2+(n-2)s}\leq p<\fz$ and $1\leq q\leq\frac{(1+(n-1)s)p}{1+(n-1)s+(s-1)p}$. Then, for every $\frac{1+(n-1)s}{2+(n-2)s}\leq p<\fz$, we define an extension operator $E_2$ on $W^{1, p}(\boz^n_\psi)$ by setting
\[E_2(u)(x):=\widetilde E_2\lf(u\circ\mathcal O^{-1}\r)(\mathcal O(x))\]
for every function $u\in W^{1, p}(\boz^n_\psi)$ and $x\in\rn$. By the same reason as above, $E_2:W^{1, p}(\boz^n_\psi)\to W^{1, q}(\rn)$ is a bounded linear extension operator whenever $\frac{1+(n-1)s}{2+(n-2)s}\leq p<\fz$ and $1\leq q\leq\frac{(1+(n-1)s)p}{1+(n-1)s+(s-1)p}$. 

\textbf{$(2):$} By an argument similar with the first case $(1)$, we obtain that if the original cuspidal function $\psi$ satisfies the integrability condition (\ref{eq:inc2}) then the corresponding Lipschitz cuspidal function $\hat\psi$ also satisfies the integrability condition (\ref{eq:inc2}). Hence, we have shown that $\hat\psi$ satisfies all the assumptions of the theorem from \cite[page 308]{Mazya}. This theorem tells us that there exists a bounded linear extension operator $\widetilde E_3:W^{1, p}(\boz^n_{\hat\psi})\to W^{1, q}(\rn)$ whenever $\frac{(n-1)^2s+(n-1)}{n}\leq p<\fz$ and $1\leq q\leq n-1$. For every $\frac{(n-1)^2s+(n-1)}{n}\leq p<\fz$, we define an extension operator $E_3$ on $W^{1, p}(\boz^n_\psi)$ by setting \[E_3(u)(x):=\widetilde E_3\lf(u\circ\mathcal O^{-1}\r)(\mathcal O(x))\]
for every function $u\in W^{1, p}(\boz^n_\psi)$ and $x\in\rn$. By the same reason as above, $E_3:W^{1, p}(\boz^n_\psi)\to W^{1, q}(\rn)$ is a bounded linear extension operator for every $\frac{(n-1)^2s+(n-1)}{n}\leq p<\fz$ and $1\leq q\leq n-1$.

Next, let us show the necessity of integrability conditions (\ref{eq:inc1}) and (\ref{eq:inc2}). First, let us show that if the doubling condition (\ref{eq:doubling}) holds for a cuspidal function $\psi$, it also holds for its corresponding Lipschitz cuspidal function $\hat\psi$. Since $\hat\psi$ is Lipschitz and increasing, it suffices to show that there exists a positive constant $C>1$ such that for every $\hat t\in\lf(0, \frac{1}{2(1+\psi(1))}\r]$, we have 
\begin{equation}\label{eq:Doub}
\hat\psi(2\hat t)\leq C\hat\psi(\hat t).
\end{equation}
Let $\hat t\in\lf(0, \frac{1}{2(1+\psi(1))}\r]$ be arbitrary. There exists a unique pair $(t_{\hat t}, r_{\hat t})$ with 
\begin{equation}\label{eq:Ineq1}
\psi(t_{\hat t})\leq r_{\hat t}\leq\lim_{s\to t_{\hat t}^+}\psi(s)
\end{equation}
and 
\begin{equation}\label{eq:Ineq2}
t_{\hat t}+r_{\hat t}=(1+\psi(1))\hat t.
\end{equation}
Moreover, there exists a unique pair $(t_{2\hat t}, r_{2\hat t})$ with 
\begin{equation}\label{eq:Ineq3}
\psi(t_{2\hat t})\leq r_{2\hat t}\leq\lim_{s\to t^+_{2\hat t}}\psi(s)
\end{equation}
and 
\begin{equation}\label{eq:Ineq4}
t_{2\hat t}+r_{2\hat t}=(1+\psi(1))2\hat t.
\end{equation}
By the definition in Lemma \ref{le:LipcF}, we have 
\[\hat\psi(2\hat t)=r_{2\hat t}\ {\rm and}\ \hat\psi(\hat t)=r_{\hat t}.\]
If $r_{2\hat t}\leq 2 r_{\hat t}$, then (\ref{eq:Doub}) holds with $C=2$. Hence, we assume $r_{2\hat t}>2r_{\hat t}$. By (\ref{eq:Ineq2}) and (\ref{eq:Ineq4}), we have $t_{2\hat t}<2 t_{\hat t}$. Since $\psi$ is increasing and satisfies inequality (\ref{eq:doubling}), by (\ref{eq:Ineq1}) and (\ref{eq:Ineq3}), we have 
\[r_{2\hat t}\leq\psi(2t_{\hat t})\leq C\psi(t_{\hat t})\leq Cr_{\hat t}.\]
We have showed inequality (\ref{eq:Doub}) holds for every $\hat t\in\lf(0, \frac{1}{2(1+\psi(1))}\r]$. The fact that $\boz^n_\psi$ is globally bi-Lipschitz equivalent to $\boz^n_{\hat\psi}$ implies that $\boz^n_\psi$ and $\boz^n_{\hat\psi}$ have the same Sobolev extension properties. By the results due to Maz'ya and Poborch in \cite[pages 304 and 312]{Mazya}, if $\boz^n_{\hat\psi}$ is a Sobolev $(p, q)$-extension domain with \[\lf(\frac{1+(n-1)s}{n}\leq p<\fz, 1\leq q\leq\frac{np}{1+(n-1)s}\r)\] or \[\lf(\frac{1+(n-1)s}{2+(n-2)s}\leq p<\fz, 1\leq q\leq\frac{(1+(n-1)s)p}{1+(n-1)s+(s-1)p}\r),\] then $\hat\psi$ satisfies the integrability condition
\[\int_0^1\lf(\frac{t^s}{\hat\psi(t)}\r)^{\frac{n}{s-1}}\frac{dt}{t}<\fz.\] 
By the theorem from \cite[page 308]{Mazya}, if $\boz^n_{\hat\psi}$ is a Sobolev $(p, q)$-extension domain with $\frac{(n-1)^2s+(n-1)}{n}\leq p<\fz$ and $1\leq q\leq n-1$, then $\hat\psi$ satisfies the integrability condition
\[\int_0^1\lf(\frac{t^s}{\hat\psi(t)}\r)^{\frac{n}{s-1}}\left|\log\lf(\frac{\hat\psi(t)}{t}\r)\right|^{-\alpha}\frac{dt}{t}<\fz\]
for $\alpha=\frac{(n-2)p}{p+1-n}$. By the definition of $\hat\psi$, for every $t\in(0, 1]$, we always have 
\[\psi(t)\geq\hat\psi\lf(\frac{1}{1+\psi(1)}t\r).\]
Hence, if the corresponding Lipschitz cuspidal function $\hat\psi$ satisfies the integrability conditions (\ref{eq:inc1}) and (\ref{eq:inc2}), the original cuspidal function $\psi$ also satisfies them.
\end{proof}}

Let us prove Theorem \ref{thm:main}.
\begin{proof}[Proof of Theorem \ref{thm:main}]
As we mentioned, every outward cuspidal domain is a Sobolev $(\fz, \fz)$-extension domain. Hence, it suffices to deal with the case $1\leq q\leq p<\fz$. We first prove the result for Lipschitz outward cuspidal domains and then extend the result to arbitrary outward cuspidal domains via the global bi-Lipschitz equivalence method established in Proposition \ref{prop:bilip}.

Let $1\leq q<n-1$ and $(n-1)q/(n-1-q)\leq p<\fz$ be fixed. Let $\psi$ be a Lipschitz cuspidal function. We define a cylinder $\widehat{\mathsf C}_o$ by setting
\begin{equation}\label{eq:cylin}
\widehat{\mathsf C}_o:=\{(t, x)\in[1, 3)\times\rr^{n-1}: |x|<2\psi(1)\}.
\end{equation}
Then we define two sub-cylinders $\widehat{\mathsf C}_o^1$ and $\widehat{\mathsf C}_o^2$ of $\widehat{\mathsf C}_o$ by setting
\[\widehat{\mathsf C}_o^1:=\{(t, x)\in(1, 2)\times\rr^{n-1}:|x|<2\psi(1)\}\]
and 
\[\widehat{\mathsf C}_o^2:=\{(t, x)\in(2, 3)\times\rr^{n-1}:|x|<2\psi(1)\}.\]
We also define a sub-cylinder $\mathsf C_o^1$ of $\widehat{\mathsf C}_o^1$ by setting 
\[\mathsf C_o^1:=\{(t, x)\in(1, 2)\times\rr^{n-1}: |x|<\psi(1)\}.\]
Then $A_{\mathsf C_o^1}:=\widehat{\mathsf C}_o^1\setminus\overline{\mathsf C}_o^1$ is an annular set. We define a reflection $\widetilde{\mathcal R_1}: A_{\mathsf C_o^1}\to\mathsf C_o^1$ by setting
\begin{equation}\label{eq:ref0}
\widetilde{\mathcal R_1}(t, x):=\lf(t, \lf(\frac{3}{2}\psi(1)-\frac{|x|}{2}\r)\frac{x}{|x|}\r).
\end{equation}
There exists a positive constant $C$ such that for every $(t, x)\in A_{\mathsf C_o^1}$, we have 
\begin{equation}\label{eq:est1}
|D\widetilde{\mathcal R_1}(t, x)|\leq C\ {\rm and}\ \frac{1}{C}\leq|J_{\widetilde{\mathcal R_1}}(t, x)|\leq C.
\end{equation}
We also define a cut-off function $\widetilde L_1$ on the annular set $\overline{A_{\mathsf C_o^1}}$ with $\widetilde L_1\equiv 0$ on $[1, 2]\times\partial B^{n-1}(0, 2\psi(1))$ and $\widetilde L_1\equiv 1$ on $[1, 2]\times\partial B^{n-1}(0, \psi(1))$ by setting 
\begin{equation}\label{eq:cutoff1}
\widetilde L_1(t, x):=2-\frac{|x|}{\psi(1)}.
\end{equation}
There exists a positive constant $C$, such that for every $(t, x)\in A_{\mathsf C_o^1}$, we have 
\begin{equation}\label{eq:bond1}
|\nabla\widetilde L_1(t, x)|\leq C
\end{equation}
Next, we define a reflection $\widetilde{\mathcal R_2}:\widehat{\mathsf C}_o^2\to\widehat{\mathsf C}_o^1$ by setting 
\begin{equation}\label{eq:Ref1}
\widetilde{\mathcal R_2}(t, x):=(t-2, x).
\end{equation} 
There exists a positive constant $C$, such that for every $(t, x)\in\widehat{\mathsf C}_o^2$, we have
\begin{equation}\label{eq:bond2}
|D\widetilde{\mathcal R_2}(t, x)|\leq C\ {\rm and}\ \frac{1}{C}\leq|J_{\widetilde{\mathcal R_2}}(t, x)|\leq C.
\end{equation}
Moreover we define a cut-off function $\widetilde L_2$ on $\overline{\widehat{\mathsf C}_o^2}$  with $\widetilde L_2\equiv 1$ on $\{2\}\times\overline{B^{n-1}(0, 2\psi(1))}$ and $\widetilde L_2\equiv 0$ on $\{3\}\times\overline{B^{n-1}(0, 2\psi(1))}$ by setting
\begin{equation}\label{eq:cutoff2}
\widetilde L_2(t, x):=3-t.
\end{equation}
Then, for every $(t, x)\in\widehat{\mathsf C}_o^2$, we have 
\begin{equation}\label{eq:bond3}
|\nabla\widetilde L_2(t, x)|\leq 2.
\end{equation}
Next, we define a double outward cuspidal domain $\widehat{\boz_\psi^n}$ by setting 
\[\widehat{\boz_\psi^n}:=\{(t, x)\in(0, 1]\times\rr^{n-1}: |x|<2\psi(t)\}\cup\widehat{\mathsf C}_o.\]
We will construct a bounded linear extension operator $E$ from $W^{1,p}(\boz_\psi^n)$ to $W^{1,q}(\rn)$ such that for every function $u\in W^{1.p}(\boz_\psi^n)$, we have $E(u)=0$ on $\partial\widehat{\boz_\psi^n}\setminus\{0\}$. We define an annular-type set by setting
$$A_\psi^n:=\{(t, x)\in(0, 1]\times\rr^{n-1}: \psi(t)<|x|<2\psi(t)\}.$$
Moreover we define a reflection $\mathcal R: A_\psi^n\to\boz_\psi^n$ by setting 
\begin{equation}\label{eq:3ref}
\mathcal R(z)=\mathcal R(t, x):=\lf(t, \lf(\frac{-|x|}{2}+\frac{3}{2}\psi(t)\r)\frac{x}{|x|}\r).
\end{equation}
Since $\psi$ is Lipschitz, there exists a positive constant $C$ such that for every $z=(t, x)\in A_\psi^n$, we have 
\begin{equation}\label{eq:Esref}
|D\mathcal R(z)|\leq C\ {\rm and}\ \frac{1}{C}\leq\lf|J_{\mathcal R}(z)\r|\leq C.
\end{equation}
We define a cut-off function $L$ on $A_\psi^n$ by setting 
\begin{equation}\label{eq:Cutoff}
L(t, x):=\frac{-|x|}{\psi(t)}+2.
\end{equation}
Since $\psi$ is Lipschitz and 
$$\psi(t)<|x|<2\psi(t)$$
for every point $(t, x)\in A_\psi^n$, there exists a positive constant $C$ such that 
\begin{equation}\label{eq:HP1}
\lf|\nabla L(t, x)\r|\leq\frac{C}{\psi(t)}\ {\rm for\ almost\ every}\ (t, x)\in A_\psi^n.
\end{equation}

By Lemma \ref{lem:density}, $C_o^\fz(\rn)\cap W^{1, p}(\boz_\psi^n)$ is dense in $W^{1, p}(\boz_\psi^n)$. We first define a bounded linear extension operator from the dense subspace $C_o^\fz(\rn)\cap W^{1,p}(\boz_\psi^n)$ to $W^{1,q}(\rn)$ and then extend it to the full space $W^{1,p}(\boz_\psi^n)$. Let $u\in C_o^\fz(\rn)\cap W^{1,p}(\boz_\psi^n)$ be fixed. We define an extension $E(u)$ on $\widehat{\boz_\psi^n}$ by setting 
\begin{equation}\label{eq:ExtenO}
E(u)(z):=\begin{cases}
u(z), & z\in\overline{\boz_\psi^n}, \\
L(z)(u\circ\mathcal R)(z), & z\in A_\psi^n,\\
\widetilde L_1(z)(u\circ\widetilde{\mathcal R}_1)(z), & z\in A_{\mathsf C_o^1}, \\
\widetilde L_2(z)(E(u)\circ\widetilde{\mathcal R}_2)(z), & z\in \widehat{\mathsf C}_o^2\, .
\end{cases}
\end{equation}
Then, $E(u)$ is continuous on $\widehat{\boz_\psi^n}$ with $E(u)\equiv 0$ on $\partial\widehat{\boz_\psi^n}\setminus\{0\}$. Hence, we can simply extend $E(u)$ to be zero outside the domain $\widehat{\boz_\psi^n}$. First, let us estimate the $L^q$-norm of $E(u)$. By the definition of $E(u)$, the H\"older inequality implies 
\begin{equation}\label{eq:lpno1}
\lf(\int_{\boz_\psi^n}|E(u)(z)|^qdz\r)^{\frac{1}{q}}\leq C\lf(\int_{\boz_\psi^n}|u(z)|^pdz\r)^{\frac{1}{p}}.
\end{equation}
By (\ref{eq:bond2}) and the fact that $0\leq\widetilde L_2\leq 1$ on $\widehat{\mathsf C}_o^2$, the change of variables formula implies
\begin{equation}\label{eq:lpno2}
\int_{\widehat{\mathsf C}_o^2}|E(u)(z)|^qdz\leq C\int_{\widehat{\mathsf C}_o^2}\lf|\lf(E(u)\circ\widetilde{\mathcal R}_2\r)(z)\r|^qdz\leq C \int_{\widehat{\mathsf C}_o^1}|E(u)(z)|^qdz.
\end{equation}
By (\ref{eq:est1}) and the fact that $0\leq\widetilde L_1\leq 1$ on $A_{\mathsf C_o^1}$, the change of variables formula and the H\"older inequality imply 
\begin{multline}\label{eq:lpno3}
\lf(\int_{\widehat{\mathsf C}_o^1}|E(u)(z)|^qdz\r)^{\frac{1}{q}}\leq C\lf(\int_{\mathsf C_o^1}\lf|E(u)(z)\r|^qdz\r)^{\frac{1}{q}}\\
+C\lf(\int_{A_{\mathsf C_o^1}}\lf|E(u)(z)\r|^qdz\r)^{\frac{1}{q}}\leq C\lf(\int_{\boz_\psi^n}|u(z)|^pdz\r)^{\frac{1}{p}}.
\end{multline}
By combining (\ref{eq:lpno2}) and (\ref{eq:lpno3}), we obtain 
\begin{equation}\label{eq:lpno4}
\lf(\int_{\widehat{\mathsf C}_o}|E(u)(z)|^qdz\r)^{\frac{1}{q}}\leq C\lf(\int_{\boz_\psi^n}|u(z)|^pdz\r)^{\frac{1}{p}}.
\end{equation}
By (\ref{eq:Esref}) and the fact that $0\leq L(z)\leq 1$ on $A_\psi^n$, the change of variables formula and the H\"older inequality imply 
\begin{equation}\label{eq:Hp1}
\lf(\int_{A_\psi^n}|L(z)(u\circ\mathcal R)(z)|^qdz\r)^{\frac{1}{q}}\leq \lf(\int_{\boz_\psi^n}|u(z)|^pdz\r)^{\frac{1}{p}}.
\end{equation}
Combining inequalities (\ref{eq:lpno1}), (\ref{eq:lpno4}) and (\ref{eq:Hp1}), we obtain 
\begin{equation}\label{eq:Hp2}
\lf(\int_{\widehat{\boz_\psi^n}}|E(u)(z)|^qdz\r)^{\frac{1}{q}}\leq C\lf(\int_{\boz_\psi^n}|u(z)|^pdz\r)^{\frac{1}{p}}.
\end{equation}

Now, let us estimate the $L^q$-norm of $\lf|\nabla E(u)\r|$. First, the H\"older inequality implies
\begin{equation}\label{eq:Whp1}
\lf(\int_{\boz_\psi^n}|\nabla E(u)(z)|^qdz\r)^{\frac{1}{q}}\leq C\lf(\int_{\boz_\psi^n}|\nabla u(z)|^pdz\r)^{\frac{1}{p}}.
\end{equation} 
The chain rule implies that for almost every $z\in A_{\mathsf C_o^1}$, we have 
\begin{equation}\label{eq:spno30}
\lf|\nabla E(u)(z)\r|\leq\lf|\nabla\tilde L_1(z)(u\circ\widetilde{\mathcal R}_1)(z)\r|+\lf|\tilde L_1(z)\nabla(u\circ\widetilde{\mathcal R}_1)(z)\r|.
\end{equation}
By (\ref{eq:est1}) and (\ref{eq:bond1}), the change of variables formula and the H\"older inequality imply
\begin{equation}\label{eq:spno31}
\lf(\int_{A_{\mathsf C_o^1}}\lf|\nabla\tilde L_1(z)(u\circ\widetilde{\mathcal R}_1)(z)\r|^qdz\r)^{\frac{1}{q}}\leq C\lf(\int_{\boz_\psi^n}\lf|u(z)\r|^pdz\r)^{\frac{1}{p}}.
\end{equation}
By (\ref{eq:est1}) and the fact that $0\leq\tilde L_1(z)\leq 1$ for every $z\in A_{\mathsf C_o^1}$, the change of variables formula and the H\"older inequality imply 
\begin{equation}\label{eq:spno32}
\lf(\int_{A_{\mathsf C_o^1}}\lf|\tilde L_1(z)\nabla(u\circ\widetilde{\mathcal R}_1)(z)\r|^qdz\r)^{\frac{1}{q}}\leq C\lf(\int_{\boz_\psi^n}\lf|\nabla u(z)\r|^pdz\r)^{\frac{1}{p}}.
\end{equation}
By combining (\ref{eq:spno30}), (\ref{eq:spno31}) and (\ref{eq:spno32}), we obtain 
\begin{equation}\label{eq:spno33}
\lf(\int_{A_{\mathsf C_o^1}}\lf|\nabla E(u)(z)\r|^qdz\r)^{\frac{1}{q}}\leq C\lf(\int_{\boz_\psi^n}\lf|u(z)\r|^p+\lf|\nabla u(z)\r|^pdz\r)^{\frac{1}{p}}.
\end{equation}
The chain rule implies that for almost every $z\in\widehat{\mathsf C}_o^2$, we have 
\begin{equation}\label{eq:spno34}
\lf|\nabla E(u)(z)\r|\leq\lf|\nabla\tilde L_2(z)\lf(E(u)\circ\widetilde{\mathcal R}_2\r)(z)\r|+\lf|\tilde L_2(z)\nabla\lf(E(u)\circ\widetilde{\mathcal R}_2\r)(z)\r|.
\end{equation}
By (\ref{eq:bond2}), (\ref{eq:bond3}) and (\ref{eq:lpno3}), the change of variables formula and the H\"older inequality imply
\begin{multline}\label{eq:spno35}
\lf(\int_{\widehat{\mathsf C}_o^2}\lf|\nabla\tilde L_2(z)\lf(E(u)\circ\widetilde{\mathcal R}_2\r)(z)\r|^qdz\r)^{\frac{1}{q}}\leq C\lf(\int_{\widehat{\mathsf C}_0^1}\lf|E(u)(z)\r|^qdz\r)^{\frac{1}{q}}\\
\leq C\lf(\int_{\boz_\psi^n}\lf|u(z)\r|^pdz\r)^{\frac{1}{p}}.
\end{multline}
By (\ref{eq:bond2}) and the fact that $0\leq\tilde L_2(z)\leq 1$ for almost every $z\in\widehat{\mathsf C}_o^2$, the change of variables formula, (\ref{eq:Whp1}) and (\ref{eq:spno33}) imply 
\begin{multline}\label{eq:spno36}
\lf(\int_{\widehat{\mathsf C}_o^2}\lf|\tilde L_2(z)\nabla\lf(E(u)\circ\widetilde{R}_2\r)(z)\r|^qdz\r)^{\frac{1}{q}}\leq C\lf(\int_{\widehat{\mathsf C}_o^1}\lf|\nabla E(u)(z)\r|^qdz\r)^{\frac{1}{q}}\\
\leq C\lf(\int_{A_{\mathsf C_o^1}}\lf|\nabla E(u)(z)\r|^qdz\r)^{\frac{1}{q}}+\lf(\int_{\mathsf C_o^1}\lf|\nabla E(u)(z)\r|^qdz\r)^{\frac{1}{q}}\\
\leq C\lf(\int_{\boz_\psi^n}\lf|u(z)\r|^p+\lf|\nabla u(z)\r|^pdz\r)^{\frac{1}{p}}.
\end{multline}
Hence, by combining (\ref{eq:spno34}), (\ref{eq:spno35}) and (\ref{eq:spno36}), we obtain 
\begin{equation}\label{eq:spno37}
\lf(\int_{\widehat{\mathsf C}^2_o}\lf|\nabla E(u)(z)\r|^qdz\r)^{\frac{1}{q}}\leq C\lf(\int_{\boz_\psi^n}\lf|u(z)\r|^p+\lf|\nabla u(z)\r|^pdz\r)^{\frac{1}{p}}.
\end{equation}
Since $\mathsf C_o^1\subset\boz_\psi^n$, by combining (\ref{eq:Whp1}), (\ref{eq:spno33}) and (\ref{eq:spno37}), we obtain 
\begin{equation}\label{eq:spno38}
\lf(\int_{\widehat{\mathsf C}_o}\lf|\nabla E(u)(z)\r|^qdz\r)^{\frac{1}{q}}\leq C\lf(\int_{\boz_\psi^n}\lf|u(z)\r|^p+\lf|\nabla u(z)\r|^pdz\r)^{\frac{1}{p}}.
\end{equation}
The chain rule implies that for almost every $z\in A_\psi^n$, we have 
\begin{equation}\label{eq:Hp3}
\lf|\nabla E(u)(z)\r|\leq \lf|\nabla L(z)(u\circ\mathcal R)(z)\r|+\lf|L(z)\nabla(u\circ\mathcal R)(z)\r|.
\end{equation}
By (\ref{eq:Esref}) and the fact $0\leq L(z)\leq 1$ on $A_\psi^n$, the change of variables formula and the H\"older inequality yield 
\begin{equation}\label{eq:Hp4}
\lf(\int_{A_\psi^n}\lf|L(z)\nabla(u\circ\mathcal R)(z)\r|^qdz\r)^{\frac{1}{q}}\leq \lf(\int_{\boz_\psi^n}|\nabla u(z)|^pdz\r)^{\frac{1}{p}}.
\end{equation} 
The H\"older inequality implies 
\begin{equation}\label{eq:Hp5}
\lf(\int_{A_\psi^n}\lf|\nabla L(z)(u\circ\mathcal R)(z)\r|^qdz\r)^{\frac{1}{q}}\leq\lf(\int_{A_\psi^n}|\nabla L(z)|^{\frac{pq}{p-q}}dz\r)^{\frac{p-q}{pq}}\cdot\lf(\int_{A_\psi^n}|(u\circ\mathcal R)(z)|^pdz\r)^{\frac{1}{p}}.
\end{equation}
By (\ref{eq:Esref}), the change of variables formula implies 
\begin{equation}\label{eq:Hp6}
\int_{A_\psi^n}|(u\circ\mathcal R)(z)|^pdz\leq C\int_{\boz_\psi^n}|u(z)|^pdz.
\end{equation}
By (\ref{eq:HP1}) and the fact that $\psi$ is Lipschitz, we have 
\begin{equation}\label{eq:Hp7}
\int_{A_\psi^n}|\nabla L(z)|^{\frac{pq}{p-q}}\leq C\int_0^1\psi(t)^{n-1-\frac{pq}{p-q}}dt<C\int_0^1t^{n-1-\frac{pq}{p-q}}dt<\fz
\end{equation}
whenever $1\leq q<n-1$ and $(n-1)q/(n-1-q)\leq p<\fz$. By combining inequalities (\ref{eq:Hp5}), (\ref{eq:Hp6}) and (\ref{eq:Hp7}), we obtain 
\begin{equation}\label{eq:Hp8}
\lf(\int_{A_\psi^n}\lf|\nabla L(z)(u\circ\mathcal R)(z)\r|^qdz\r)^{\frac{1}{q}}\leq C\lf(\int_{\boz_\psi^n}|u(z)|^pdz\r)^{\frac{1}{p}}.
\end{equation}
By combining (\ref{eq:Hp3}), (\ref{eq:Hp4}) and (\ref{eq:Hp8}), we obtain 
\begin{equation}\label{eq:Hp9}
\lf(\int_{A_\psi^n}\lf|\nabla E(u)(z)\r|^qdz\r)^{\frac{1}{q}}\leq C\lf(\int_{\boz_\psi^n}|u(z)|^p+|\nabla u(z)|^pdz\r)^{\frac{1}{p}}.
\end{equation}
By combining inequalities (\ref{eq:Whp1}), (\ref{eq:spno37}) and (\ref{eq:Hp9}), we obtain 
\begin{equation}\label{eq:Hp10}
\lf(\int_{\widehat{\boz_\psi^n}}|\nabla E(u)(z)|^qdz\r)^{\frac{1}{q}}\leq C\lf(\int_{\boz_\psi^n}|u(z)|^p+|\nabla u(z)|^pdz\r)^{\frac{1}{p}}.
\end{equation}
Finally, by combining (\ref{eq:Hp2}) and (\ref{eq:Hp10}), we obtain the desired norm inequality
\begin{equation}\label{eq:Hp11}
\lf(\int_{\widehat{\boz_\psi^n}}\lf|E(u)(z)\r|^q+\lf|\nabla E(u)(z)\r|^qdz\r)^{\frac{1}{q}}\leq C\lf(\int_{\boz_\psi^n}|u(z)|^p+|\nabla u(z)|^pdz\r)^{\frac{1}{p}}.
\end{equation}
Hence, the extension operator $E$ defined in (\ref{eq:ExtenO}) is a bounded linear extension operator from the dense subspace $C_o^\fz(\rn)\cap W^{1,p}(\boz_\psi^n)$ to $W^{1,q}(\rn)$ whenever $1\leq q<n-1$ and $(n-1)q/(n-1-q)\leq p<\fz$. For an arbitrary $u\in W^{1,p}(\boz_\psi^n)$, there exists a Cauchy sequence $\{u_m\}_{m=1}^\fz\subset C_o^\fz(\rn)\cap W^{1,p}(\boz_\psi^n)$ which converges to $u$ with respect to the $W^{1,p}$-norm. Since $E$ is a  bounded linear extension operator from $C_o^\fz(\rn)\cap W^{1,p}(\boz_\psi^n)$ to $W^{1,q}(\rn)$, there exists a subsequence of $\{u_m\}$ which converges to $u$ almost everywhere on $\boz_\psi^n$ and $\{E(u_m)\}$ is also a Cauchy sequence in $W^{1,q}(\rn)$ which converges to some function $v\in W^{1,q}(\rn)$. To simplify the notation, we still denote this subsequence by $\{u_m\}$. Then, we have $v\big|_{\boz_\psi^n}(z)=u(z)$ for almost every $z\in\boz_\psi^n$ and 
\[\|v\|_{W^{1,q}(\rn)}\leq\lim_{m\to\fz}\|E(u_m)\|_{W^{1,q}(\rn)}\leq C\lim_{m\to\fz}\|u_m\|_{W^{1,p}(\boz_\psi^n)}\leq C\|u\|_{W^{1,p}(\boz_\psi^n)}.\]
Furthermore, by picking an extra subsequence if necessary, it follows that $v(z)=E(u)(z)$ for almost every $z\in\rn$. Hence, $E$ is a bounded linear extension operator from $W^{1,p}(\boz_\psi^n)$ to $W^{1,q}(\rn)$, whenever $1\leq q<n-1$ and $(n-1)q/(n-1-q)\leq p<\fz$. In conclusion, we have proved that a Lipschitz outward cuspidal domain $\boz_\psi^n$ is a Sobolev $(p, q)$-extension domain, whenever $1\leq q<n-1$ and $(n-1)q/(n-1-q)\leq p<\fz$.

Next, we extend the result to an arbitrary outward cuspidal domain by using global bi-Lipschitz transformations. Let $\psi:(0, 1]\to(0, \fz)$ be an arbitrary cuspidal domain. By Proposition \ref{prop:bilip}, there exists a Lipschitz cuspidal function $\hat\psi$ and a global bi-Lipschitz homeomorphism $\mathcal O:\rn\to\rn$ with $\mathcal O(\boz_\psi^n)=\boz_{\hat\psi}^n$. Fix $1\leq q<n-1$ and $(n-1)q/(n-1-q)\leq p<\fz$. Let $u\in W^{1,p}(\boz_\psi^n)$ be arbitrary. Since $\mathcal O$ is bi-Lipschitz with $\mathcal O(\boz_\psi^n)=\boz_{\hat\psi}^n$, we have $u\circ\mathcal O^{-1}\in W^{1,p}(\boz_{\hat\psi}^n)$ with 
\begin{equation}\label{eq:ineq1}
\|u\circ\mathcal O^{-1}\|_{W^{1,p}(\boz_{\hat\psi}^n)}\leq C\|u\|_{W^{1, p}(\boz_\psi^n)}.
\end{equation}
By the argument above, the Lipschitz outward cuspidal domain $\boz_{\hat\psi}^n$ is a Sobolev $(p, q)$-extension domain, there exists a function $E(u\circ\mathcal O^{-1})\in W^{1,q}(\rn)$ with $E(u\circ\mathcal O^{-1})\big|_{\boz_{\hat\psi}^n}\equiv u\circ\mathcal O^{-1}$ and
\begin{equation}\label{eq:ineq2}
\|E(u\circ\mathcal O^{-1})\|_{W^{1,q}(\rn)}\leq C\|u\circ\mathcal O^{-1}\|_{W^{1,p}(\boz_{\hat\psi}^n)}.
\end{equation}
By the fact that $\mathcal O$ is bi-Lipschitz, we have $(E(u\circ\mathcal O^{-1}))\circ\mathcal O\in W^{1,q}(\rn)$ with
\begin{equation}\label{eq:ineq3}
\|(E(u\circ\mathcal O^{-1}))\circ\mathcal O\|_{W^{1,q}(\rn)}\leq C\|E(u\circ\mathcal O^{-1})\|_{W^{1,q}(\rn)}.
\end{equation}
By the definitions, we have $(E(u\circ\mathcal O^{-1}))\circ\mathcal O\big|_{\boz_\psi^n}\equiv u$. By combining (\ref{eq:ineq1}), (\ref{eq:ineq2}) and (\ref{eq:ineq3}), we obtain the desired norm inequality 
\[\|(E(u\circ\mathcal O^{-1}))\circ\mathcal O\|_{W^{1,q}(\rn)}\leq C\|u\|_{W^{1, p}(\boz_\psi^n)}.\]
Hence, $\boz_\psi^n$ is a Sobolev $(p, q)$-extension domain whenever $1\leq q<n-1$ and $(n-1)q/(n-1-q)\leq p<\fz$.
\end{proof}

Let us prove Proposition \ref{proposition} which will show the sharpness of Theorem \ref{thm:main}.
\begin{proof}[Proof of Proposition \ref{proposition}]
Let $n-1\leq p<\fz$. Set 
\[s_1:=\frac{np-(n-1)}{(n-1)^2}.\] 
By Theorem \ref{prop:mazya}, $\boz^n_{t^s}$ is a Sobolev $(p, n-1)$-extension domain if and only if $1\leq s\leq s_1$. Let $1\leq q<n-1$ and $q\leq p<\frac{(n-1)q}{n-1-q}$. Set 
\[s_2:=\frac{pq+p-q}{pq+(n-1)(q-p)}.\]
By Theorem \ref{prop:mazya}, $\boz^n_{t^s}$ is a Sobolev $(p, q)$-extension domain if and only if $1\leq s<s_2$.
\end{proof}

\section{Further comments}
In the monograph \cite{Mazya} and their papers referred to therein, Maz'ya and Poborchi also dealt with  generalized outward cuspidal domains with Lipschitz base domains. To be more precise, for a bounded Lipschitz domain $U\in\rr^{n-1}$ with $0\in U$ and a cuspidal function $\psi:(0, 1]\to(0, \fz)$, the corresponding generalized outward cuspidal domain with the base domain $U$ is defined by setting
\begin{equation}\label{eq:Lipcusp}
U_\psi^n:=\lf\{(t, x)\in(0, 1]\times\rr^{n-1}: x\in \psi(t)U\r\}\cup\lf\{(t, x)\in(1, 2)\times\rr^{n-1}: x\in\psi(1)U\r\}.
\end{equation} 
We have only discussed outward cuspidal domains whose base domains are the unit ball. Maz'ya and Poborchi established results for generalized Lipschitz outward cuspidal domains $U_\psi^n$ with Lipschitz base domains $U\subset\rr^{n-1}$. Hence, our results in Theorem \ref{prop:mazya} only extend Maz'ya and Poborchi's results in a special case. For a full extension, one would need to establish a global bi-Lipschitz equivalence of Proposition \ref{prop:bilip} analog for general outward cuspidal domains $U_\psi^n$. This appears to be technically challenging but we expect it to be doable.
\begin{conjecture}
For every generalized outward cuspidal domain $U_\psi^n$, there exists a Lipschitz cuspidal function $\hat\psi$, a Lipschitz base domain $\hat U$ and a global bi-Lipschitz transformation $\mathcal O:\rr^n\to\rr^n$ with $\mathcal O(U_\psi^n)=\hat U_{\hat\psi}^n$.
\end{conjecture}

\end{document}